\documentclass[12pt, reqno]{amsart}
\usepackage{amsmath, amsthm, amscd, amsfonts, amssymb, graphicx, color}
\usepackage[bookmarksnumbered, colorlinks, plainpages]{hyperref}
\hypersetup{colorlinks=true,linkcolor=red, anchorcolor=green, citecolor=cyan, urlcolor=red, filecolor=magenta, pdftoolbar=true}

\usepackage{tensor}

\newtheorem{theorem}{Theorem}[section]

\newtheorem{proposition}[theorem]{Proposition}

\theoremstyle{definition}
\newtheorem{definition}[theorem]{Definition}
\newtheorem{example}[theorem]{Example}

\theoremstyle{remark}
\newtheorem{remark}[theorem]{Remark}
\numberwithin{equation}{section}

\usepackage{fullpage}

\begin{document}

\setcounter{page}{1}

\title[FPT on $G$-metric spaces]{Fixed point theorems in $G$-metric space.}

\author[Ya\'e Olatoundji Gaba]{Ya\'e Olatoundji Gaba$^{1,*}$}

\address{$^{1}$Department of Mathematics and Applied Mathematics, University of Cape Town, South Africa.}
\email{\textcolor[rgb]{0.00,0.00,0.84}{gabayae2@gmail.com
}}

\subjclass[2010]{Primary 47H05; Secondary 47H09, 47H10.}

\keywords{$G$-metric, fixed point, orbitally continuous.}

\begin{abstract}
In this article, we present a new type of fixed point for single valed mapping in a $G$-complete $G$-metric space.
\end{abstract} 

\maketitle

\section{Introduction and Preliminaries}

The importance of fixed point in mathematical analysis and topology is no longer to be established. For instance, it is used to determine existence and uniqueness of solutions of differential and integral equations. Initially stated in the metric space setting, fixed point theory has found its way in more general spaces, even though most of them are metric-like. One of these general spaces, space of interest for our study, is the $G$-metric space, where many fixed point theorems have already been etablished, see \cite{Gaba1,Gaba5,s,Mustafa,mustafa3,mustafa4}. Throughout the years, different authors proposed different types of formulations, all expressing different contractive-type conditions and most of these contractions are Picard operators and therefore lead to the uniqeness of the fixed point. However, for a given self mapping $T$ on set $X$, if the set of fixed point of $T$ is nonempty and the sequence of successive approximation for any intial point converges to a fixed point of $T$, $T$ is called a weakly Picard operator.  
 We present here new fixed point results in a $G$-complete $G$-metric space, both in the Picard and weakly Picard cases. The mappings we consider satisfy a rational type almost contaction.
 
 The elementary facts about $G$-metric spaces can be found in \cite{Mustafa}. We give here a shortened form of these prerequisites.

\begin{definition} (Compare \cite[Definition 3]{Mustafa})
Let $X$ be a nonempty set, and let the function $G:X\times X\times X \to [0,\infty)$ satisfy the following properties:
\begin{itemize}
\item[(G1)] $G(x,y,z)=0$ if $x=y=z$ whenever $x,y,z\in X$;
\item[(G2)] $G(x,x,y)>0$ whenever $x,y\in X$ with $x\neq y$;
\item[(G3)] $G(x,x,y)\leq G(x,y,z) $ whenever $x,y,z\in X$ with $z\neq y$;
\item[(G4)] $G(x,y,z)= G(x,z,y)=G(y,z,x)=\ldots$, (symmetry in all three variables);

\item[(G5)]
$$G(x,y,z) \leq [G(x,a,a)+G(a,y,z)]$$ for any points $x,y,z,a\in X$.
\end{itemize}
Then $(X,G)$ is called a \textbf{$G$-metric space}.

\end{definition}

\begin{proposition}\label{prop1} (Compare \cite[Proposition 6]{Mustafa})
Let $(X,G)$ be a $G$-metric space. Define on $X$ the metric  $d_G$ by $d_G(x,y)= G(x,y,y)+G(x,x,y)$ whenever $x,y \in X$. Then for a sequence $(x_n) \subseteq X$, the following are equivalent
\begin{itemize}
\item[(i)] $(x_n)$ is $G$-convergent to $x\in X.$

\item[(ii)] $\lim_{n,m \to \infty}G(x,x_n,x_m)=0.$

\item[(iii)]  $\lim_{n \to \infty}d_G(x_n,x)=0$.

\item[(iv)]$\lim_{n \to \infty}G(x,x_n,x_n)=0.$ 

\item[(v)]$\lim_{n \to \infty}G(x_n,x,x)=0.$ 
\end{itemize}

\end{proposition}

\begin{proposition}(Compare \cite[Proposition 9]{Mustafa})

In a $G$-metric space $(X,G)$, the following are equivalent
\begin{itemize}
\item[(i)] The sequence $(x_n) \subseteq X$ is $G$-Cauchy.

\item[(ii)] For each $\varepsilon >0$ there exists $N \in \mathbb{N}$ such that $G(x_n,x_m,x_m)< \varepsilon$ for all $m,n\geq N$.

\end{itemize}

\end{proposition}

\begin{definition} (Compare \cite[Definition 4]{Mustafa})
 A $G$-metric space $(X,G,K)$ is said to be symmetric if 
 $$G(x,y,y) = G(x,x,y), \ \text{for all } x,y \in X.$$

\end{definition}

\begin{definition} (Compare \cite[Definition 9]{Mustafa})
 A $G$-metric space $(X,G)$ is $G$-complete if every $G$-Cauchy sequence of elements of $(X,G)$ is $G$-convergent in  $(X,G)$. 

\end{definition}

We conclude this introductory part with:

\begin{definition} 
A self mapping $T$ defined on a $G$-metric space $(X,G,K)$ is said to be \textbf{orbitally continuous} if and only if $$\lim_{i\to \infty}T^{n_i}x=x^* \in X \Longrightarrow Tx^*=\lim_{i\to \infty}TT^{n_i}x.$$
\end{definition}

\section{The results}

\begin{theorem}\label{thm1}

Let $(X,G)$ be a symmetric $G$-complete $G$-metric space and $T$ be a mapping from $X$ to itself. Suppose that $T$ satisfies the following condition:

\begin{equation}\label{eq1}
G(Tx,Ty,Tz) \leq \left( \frac{G(Tx,y,z)+G(x,Ty,z) +G(x,y,Tz) }{G(x,Tx,Tx)+G(y,Ty,Ty)+G(z,Tz,Tz)+1} \right)G(x,y,z), 
\end{equation}
 for all $x,y,z\in X$. Then
 \begin{itemize}
 \item[(a)] $T$ has at least one fixed point $\xi\in X;$
 
  \item[(b)] for any $x\in X$, the sequence $\{T^nx\}$ $G$-converges to a fixed point;
 
  \item[(c)] if $\xi, y^*\in X$ are two distinct fixed points, then
   $$G(\xi, y^*,y^*) = G(\xi, \xi,y^*)\geq \frac{1}{3}.$$
 
\end{itemize}

\end{theorem}

\begin{proof}
Let $x_0\in X$ be arbitrary and construct the sequence $\{x_n\}$ such that $x_{n+1}=Tx_n.$ 

We have, for the triplet $(x_n, x_{n+1},x_{n+1})$, and by setting $d_n = G(x_n, x_{n+1},x_{n+1})$, we have:

\begin{align*}
d_n=G(x_n, x_{n+1},x_{n+1}) & = G(Tx_{n-1}, Tx_{n},Tx_{n})\\
                            & \leq \left(\frac{G(x_{n},x_{n},x_{n})+2G(x_{n-1},x_{n+1},x_{n+1})}{d_{n-1}+2d_n+1}\right) d_{n-1}\\
                             & \leq \left(\frac{2d_{n-1}+2d_{n}}{d_{n-1}+2d_n+1}\right)d_{n-1}.
\end{align*}

\newpage

If we set 

$$\alpha_n =\frac{2d_{n-1}+2d_{n}}{d_{n-1}+2d_n+1}, $$

we get, iteratively 

\begin{align*}
d_n & \leq \alpha_n d_{n-1} \\
    & \leq \alpha_n \alpha_{n-1}d_{n-2}\\
    & \vdots \\
    & \leq \alpha_n \alpha_{n-1}\cdots\alpha_1d_0.
\end{align*}

It is clear that the sequence $\{\alpha_n\}$ is a non-increasing sequence of positive reals, so 
$$\alpha_n \alpha_{n-1}\cdots\alpha_1\leq \alpha_1^n \to 0 \text{ as } n \to \infty.$$
Therefore 
\[
\lim_{n\to \infty}\alpha_n \alpha_{n-1}\cdots\alpha_1=0,
\]
hence

\[
\lim_{n\to \infty}d_n=0.
\]
For any $m,n\in \mathbb{N}, m>n$, sincce we have
\begin{align*}
G(x_n,x_{m},x_{m}) & \leq \sum_{i=0}^{m-n}  G(x_{n+i},x_{n+i+1},x_{n+i+1}), 
\end{align*}
which translate to 

\begin{align*}
G(x_n,x_{m},x_{m}) & \leq \sum_{i=0}^{m-n}d_{n+i},
\end{align*}
we obtain

\begin{align*}
G(x_n,x_{m},x_{m}) & \leq \sum_{i=0}^{m-n}[(\alpha_{n+i}\cdots\alpha_1)d_0].
\end{align*}

Put $b_k = \alpha_{k}\cdots\alpha_1$ and observe that 

\[
\lim_{k\to \infty} \frac{b_{k+1}}{b_k}=0, \text{ i.e. the series } \sum_{k=0}^{\infty} b_k <\infty,
\]
therefore
\[
\sum_{i=0}^{m-n}(\alpha_{n+i}\cdots\alpha_1)\to 0 \text{ as } m \to \infty.
\]
In other words, $\{x_n\}$ is a $G$-Cauchy sequence so $G$-converges to some $\xi\in X.$

\newpage

\underline{Claim:} $\xi$ is a fixed point of $T$.

For the triplet $(x_{n+1},T\xi,T\xi)$ in \eqref{eq1}, we get

\begin{equation}\label{eqn2}
G(x_{n+1},T\xi,T\xi) \leq \left(\frac{G(x_{n},\xi,\xi)+2G(x_{n},T\xi,\xi) }{d_n+2G(\xi,T\xi,T\xi)+1}\right)G(x_{n},T\xi,T\xi)
\end{equation}

On taking the limit on both sides of \eqref{eqn2}, we have $G(\xi,T\xi,T\xi)=0,$ thus $T\xi=\xi.$

If $\kappa$ is a fixed point of $T$ with $\kappa\neq \xi$, then 

\begin{align*}
G(\xi,\kappa,\kappa) & = G(T\xi,T\kappa,T\kappa)\\
                     & \leq [G(\xi,\kappa,\kappa)+2G(\xi,\kappa,\kappa)]G(\xi,\kappa,\kappa) \\
                     & \leq 3 [G(\xi,\kappa,\kappa)]^2.
\end{align*}

Therefore,
 \[
 G(\xi,\kappa,\kappa) = G(\xi,\xi,\kappa) \geq \frac{1}{3}.
\]
\end{proof}

\begin{example}
Let $X=\{ 0,1/2,1\}$ and let $G:X^3\to [0,\infty)$ be defined by
$$ G(0,1,1)=6=G(1,0,0), \ \   G(0,1/2,1/2)=4=G(1/2,0,0)     $$
$$ G(1/2,1,1)=5=G(1,1/2,1/2), \ \ G(0,1/2,1)=15/2$$
$$ G(x,x,x)=0 \ \forall x\in X,        $$
and $G$ is a symmetric function of its three variables.
$(X,G)$ is $G$-complete. 

Let $T:X\to X$ be defined by 
$T0=0, \quad T1/2=1/2, \quad T1=0.$

$$\ G(T0,T1/2,T1/2)=G(0,1/2,1/2)=4; \ G(T0,T1,T1)= G(0,0,0)=0; $$

$$G(T1/2,T1,T1)=G(1/2,0,0)=4; \ G(T0,T1/2,T1)= G(0,1/2,0)=4,  $$

and we have 

\begin{align*}
4 = G(T0,T1/2,T1/2)& =G(0,1/2,1/2) \\
                   & \leq  \frac{G(T0,1/2,1/2)+G(0,T1/2,1/2)+G(0,1/2,T1/2)}{G(0,T0,T0)+2G(1/2,T1/2,T1/2)+1}\\
                   & \times G(0,1/2,1/2)\\
                   & = \frac{4+4+4}{1}4 = 48.
\end{align*}

Again, 

\begin{align*}
0= G(T0,T1,T1)& =G(0,0,0) \\
               & \leq \frac{G(T0,1,1)+G(0,T1,1)+G(0,1,T1)}{G(0,T0,T0)+2 G(1,T1,T1)+1} \times G(0,1,1)\\
               & = \frac{6+6+6}{13}6.
\end{align*}

Also,

\begin{align*}
4= G(T1/2,T1,T1)& =G(1/2,0,0) \\
               & \leq \frac{G(T1/2,1,1)+G(1/2,T1,1)+G(1/2,1,T1)}{G(1/2,T1/2,T1/2)+2 G(1,T1,T1)+1}\\
               & \times G(1/2,1,1)\\
               & = \frac{6+15}{13}5.
\end{align*}

Finally,

\begin{align*}
4= G(T0,T1/2,T1)& =G(0,1/2,0) \\
               & \leq \frac{G(T0,1/2,1)+G(0,T1/2,1)+G(0,1/2,T1)}{G(0,T0,T0)+G(1/2,T1/2,T1/2)+G(1,T1,T1)+1}\\
               & \times G(0,1/2,1)\\
               & = \frac{15+4}{7}\times\frac{15}{2}.
\end{align*}

Therefore $T$ satisfies all the conditions of Theorem \ref{thm1}. Also, $T$ has two distinct fixed points $\{0,1/2\}$ and
$G(0,1/2,1/2)=G(1/2,0,0)= 4 \geq 1/3.$

\end{example}

\begin{remark}
The map $T$ defined in Theorem \ref{thm1} belongs to the category of the so-called weakly Picard operator, as the uniqueness of the fixed is not guaranteed. Moreover, one could also just require $G$ to be an arbitrary $G$-metric, i.e. not neccessarily \textbf{symmetric}.
\end{remark}

In the same style, we present the following result, in which the map $T$ leaves exactly one point of $X$ fixed. This is the Picard case.

\begin{theorem}\label{thm2}
Let $(X,G)$ be a symmetric $G$-complete $G$-metric space and $T$ be a mapping from $X$ to itself. Suppose that $T$ satisfies the following condition:

\begin{equation}\label{eq2}
G(Tx,Ty,Tz) \leq \alpha  \left( \frac{\min\{G(y,Ty,Ty),G(z,Tz,Tz)\}[1+G(x,Tx,Tx)]}{1+G(x,y,z)} \right)+\beta G(x,y,z), 
\end{equation}
 for all $x,y,z\in X$ where $\alpha,\beta$ are nonnegative reals, satisfying
 \[
\alpha+\beta <1. 
 \]
Then $T$ leaves exactly one point of $X$ fixed.
\end{theorem}

\begin{proof}
Let $x_0\in X$ be arbitrary and construct the sequence $\{x_n\}$ such that $x_{n+1}=Tx_n.$ 

We have, for the triplet $(x_n, x_{n+1},x_{n+1})$, and by setting $d_n = G(x_n, x_{n+1},x_{n+1})$, we have:

\begin{align*}
d_n & =  G(Tx_{n-1}, Tx_{n},Tx_{n})\\
                            & \leq \frac{\alpha d_n[1+d_{n-1}]}{1+d_{n-1}}+ \beta d_{n-1},
\end{align*}

i.e. 
\begin{align}\label{usual}
d_n \leq \frac{\beta}{1-\alpha}d_{n-1}.
\end{align}

By usual procedure from \eqref{usual}, since $\left(\frac{\beta}{1-\alpha}\right)<1$, it follows that $\{T^nx_0\}$ is a $G$-Cauchy sequence. By $G$-completeness of $X$, there exists $x^*\in X$ such that $T^nx_0$ $G$-converges to $x^*.$

The uniqueness of $x^*$ is given for free by the condition \eqref{eq2}.
\end{proof}

We present, without proof, the following genralisation of Theorem \ref{thm2}.

\begin{theorem}\label{cor}
Let $(X,G)$ be a symmetric $G$-complete $G$-metric space and $T$ be a mapping from $X$ to itself. Suppose that $T$ satisfies the following condition:

\begin{align}\label{eq3}
G(Tx,Ty,Tz) \leq &\ a_1 \left( \frac{G(y,Ty,Ty)[1+G(x,Tx,Tx)]}{1+G(x,y,z)} \right)\nonumber \\
 & +  a_2 \left( \frac{G(z,Tz,Tz)[1+G(x,Tx,Tx)]}{1+G(x,y,z)} \right) +a_3 G(x,y,z) 
\end{align}
 for all $x,y,z\in X$ where $a_i:=a_i(x,y,z),i=1,2,3, $ are nonnegative functions such that for arbitrary $0<\lambda_1<1$:
 \[
a_1(x,y,z)+a_2(x,y,z) + a_3(x,y,z) = \sum_{i=1}^{3}a_i(x,y,z) \leq \lambda_1. 
 \]
Then $T$ leaves exactly one point of $X$ fixed.
\end{theorem}

Another genralisation of Theorem \ref{thm2} is provided by the following:

\begin{theorem}\label{thmfin}
Let $(X,G)$ be a symmetric $G$-complete $G$-metric space where $T$ is an an orbitally continuous mapping from $X$ to itself. If it is the case that $T$ satisfies the following condition:

\begin{align}\label{eqfin}
G(Tx,Ty,Tz) \leq &\ a_1 G(x,y,z)+a_2[G(x,Tx,Tx)+G(y,Ty,Ty)+G(z,Tz,Tz)] \nonumber \\
& + a_3[ G(Tx,y,z)+ G(x,Ty,z)+G(x,y,Tz)]\nonumber \\
& + a_4\min\{G(y,Ty,Ty),G(z,Tz,Tz)\}\frac{[1+G(x,Tx,Tx)]}{1+G(x,y,z)}\nonumber \\
& + a_5G(Tx,y,z)[1+G(x,Ty,z)+G(x,y,Tz)][1+G(x,y,z)]^{-1}\nonumber \\
& + a_6G(x,y,z)[1+G(x,Tx,Tx)+G(Tx,y,z)][1+G(x,y,z)]^{-1}
\nonumber \\
& + a_7 G(Tx,y,z) 
\end{align}
 for all $x,y,z\in X$ where $a_i:=a_i(x,y,z), i=1,\cdots,7,$ are nonnegative functions such that for arbitrary $0<\lambda_1<1$:
 \[
 a_1(x,y,z)+3a_2(x,y,z)+4a_3(x,y,z)+ a_4(x,y,z) 
 +a_6(x,y,z) \leq \lambda_1, 
 \]
then $T$ leaves at least one point of $X$ fixed.
\end{theorem}

\begin{proof}
Let $x_0\in X$ be arbitrary and construct the sequence $\{x_n\}$ such that $x_{n+1}=Tx_n.$ 

We have, for the triplet $(x_n, x_{n+1},x_{n+1})$, and by setting $d_n = G(x_n, x_{n+1},x_{n+1})$, we have:

\begin{align*}
d_{n+1}  = & \  G(Tx_{n}, Tx_{n+1},Tx_{n+1})\\
     \leq & \ a_1 d_n + a_2[d_n+2d_{n+1}] + a_3[2d_n+2d_{n+1}]                    +a_4  d_{n+1}+  a_6  d_{n}.
\end{align*}

i.e. 
\begin{align}\label{byusual}
d_{n+1} \leq \frac{a_1+a_2+2a_3+a_6}{1-(2a_2+2a_3+a_4)}d_{n}.
\end{align}

By usual procedure from \eqref{byusual}, since $$a_1(x,y,z)+3a_2(x,y,z)+4a_3(x,y,z)+ a_4(x,y,z) 
 +a_6(x,y,z) <1,$$ it follows that $\{T^nx_0\}$ is a $G$-Cauchy sequence. By $G$-completeness of $X$, there exists $x^*\in X$ such that $T^nx_0$ $G$-converges to $x^*.$
Since $X$ is $G$-complete. Obviously $x^*$ is the desired fixed point by orbitally continuity of $T$.

\end{proof}

\begin{example}
Let $[a,b]$ where $1<a<b.$ Let the $G$-metric $G$ be given on $[a,b]$ as:

\[
G(x,y,z)= \max\{|x-y|,|y-z|,|z-x|\}.
\]
Let $T:X\to X$ be defiend as follows:
\[
Tx = x+ \frac{1}{x}-\frac{1}{b}.
\]
Let $1-\frac{1}{b^2}\leq \alpha<1$ and $\beta,\gamma,\delta,\lambda,\mu$ and $L$ be arbitrary nonnegative reals such that 

\[
a_1(x,y,z)\leq \alpha,\ a_2(x,y,z)\leq \beta, \ a_3(x,y,z)\leq \gamma,  \  a_4(x,y,z)\leq \delta, 
\]

and 
\[
a_5(x,y,z)\leq \lambda,\ a_6(x,y,z)\leq \mu, \ a_7(x,y,z)\leq L
\]
and
\[
\alpha + 3\beta + 4 \gamma + \mu + \delta <1.
\]
Here all the conditions of Theorem \ref{thmfin} are satisfied and it is readily seen that $b$ is a fixed point of $T$.
\end{example}

In the extension of metric fixed point theory, generalization of metric spaces via complex valued ordered metric space, were introduced. The author plans to study more thoroughly, and with examples, fixed point results in the setting of ordered $G$-metric space in another paper \cite{Gaba6} but we present here a first result of the kind.

Recall that we can define a partial order $\preceq$ on the set $\mathbb{C}$ of complex numbers by setting, for any $z_1,z_2\in \mathbb{C},$
 $$z_1 \preceq z_2 \Longleftrightarrow Re(z_1)\leq Re(z_2) \text{ and } Im(z_1)\leq Im(z_2).$$ 
 Moreover, on partial ordered $G$-metric space, the convergence of a sequence is interpreted in the canonical way, i.e.
 for a sequence $\{x_n\}\subseteq (X,G,\preceq)$ where $(X,G,\preceq)$ is a partial ordered complex valued $G$-metric space,
 
 $$x_n \ G\text{-converges to } x^* \Longleftrightarrow \forall c \in \mathbb{C}, \text{with } 0\preceq c, \ \exists n_0\in \mathbb{N}:\forall n>n_0\ G(x^*,x_n,x_n) \preceq c. $$

Similarly for $G$-Cauchy sequences. Furthermore, a self mapping $T$ defined on a partial ordered $G$-metric space $(X,G,\preceq)$ is nondecreasing if $Tx\preceq Ty$ whenever $x\preceq y$, for $x,y \in X.$

We then state the result:

\begin{theorem}\label{thmfinor}

Let $(X,G,\preceq)$ be a symmetric, $G$-complete, complex valued $G$-metric space. Assume that if $\{x_n\}$ is a nondecreasing sequence of elements of $X$ such that $x_n \ G\text{-converges to } x^*$, then $x_n \preceq x^*$ for all $n\in \mathbb{N}$.
Let $T:X\to X$ be a nondecreasing mapping such that:

\begin{align}\label{eqfinor}
G(Tx,Ty,Tz) \preceq &\ a_1 G(x,y,z)+a_2[G(x,Tx,Tx)+G(y,Ty,Ty)+G(z,Tz,Tz)] \nonumber \\
& + a_3[ G(Tx,y,z)+ G(x,Ty,z)+G(x,y,Tz)]\nonumber \\
& + a_4\min\{G(y,Ty,Ty),G(z,Tz,Tz)\}\frac{[1+G(x,Tx,Tx)]}{1+G(x,y,z)}\nonumber \\
& + a_5G(Tx,y,z)[1+G(x,Ty,z)+G(x,y,Tz)][1+G(x,y,z)]^{-1}\nonumber \\
& + a_6G(x,y,z)[1+G(x,Tx,Tx)+G(Tx,y,z)][1+G(x,y,z)]^{-1}
\nonumber \\
& + a_7 G(Tx,y,z) 
\end{align}
for all $x\preceq y\preceq z\in X$ where $a_i:=a_i(x,y,z), i=1,\cdots,7,$ are nonnegative functions such that for arbitrary $0<\lambda_1<1$:
 \[
 a_1(x,y,z)+3a_2(x,y,z)+4a_3(x,y,z)+ a_4(x,y,z) 
 +a_6(x,y,z) \leq \lambda_1. 
 \]
 
 If there exists $x_0\in X$ with $x_0\preceq Tx_0$,
then $T$ leaves at least one point of $X$ fixed.

\end{theorem}

\begin{proof}
It is very easy to see that the sequence of iterates $T^nx_0, n=1,2,\cdots,$ is nondecreasing and $G$-converges to some $x^*\in X$. Therefore $x_n \preceq x^*$ for all $n\in \mathbb{N}$.
Now applying \eqref{eqfinor} to the triplet $(x_{n+1},Tx^*,Tx^*)$ and taking the limit as $n\to \infty$, we have:

\begin{align*}
G(x^*,Tx^*,Tx^*) & = G(Tx_{n},Tx^*,Tx^*)\\
                     & \preceq 2 a_2 G(x^*,Tx^*,Tx^*)+2 a_3 G(x^*,Tx^*,Tx^*)+a_4 G(x^*,Tx^*,Tx^*)\\
                     &= [2 a_2+2 a_3+a_4]G(x^*,Tx^*,Tx^*).
\end{align*} 

Since $2 a_2+2 a_3+a_4<1$, this is a contradiction unless $G(x^*,Tx^*,Tx^*)=0,$ and hence $Tx^*=x^*.$

\end{proof}

\begin{example}
Let $X= [1.5,2]$ with the usual partial order $``\leq"$. Let the $G$-metric $G$ be given on $X$ as:

\[
G(x,y,z)= \max\{|x-y|,|y-z|,|z-x|\}(1+i).
\]
Let $T:X\to X$ be defiend as follows:

$$Tx=
\begin{cases}
1.81, \ \ \ \ \  \ \  \text{ if } \ 1.5\leq x < 1.75 ,\\
x+\frac{1}{x}-\frac{1}{2},  \text{ if }\  1.75\leq x \leq 2.
\end{cases}
$$
Let $\alpha\in \left[\frac{3}{4},1\right), \beta=\gamma=\delta=\mu=\lambda=0$ and $L\geq 3$ be arbitrary nonnegative reals such that 

\[
a_1(x,y,z)\leq \alpha,\ 0=a_2(x,y,z)\leq \beta, \ 0=a_3(x,y,z)\leq \gamma,  \  0=a_4(x,y,z)\leq \delta, 
\]

and 
\[
0=a_5(x,y,z)\leq \lambda,\ a_6(x,y,z)\leq \mu, \ a_7(x,y,z)\leq L.
\]
Here all the conditions of Theorem \ref{thmfinor} are satisfied and it is readily seen that $2$ is a fixed point of $T$.

\end{example}

\begin{remark}
In the abover theorem, one could observe that there is no need of imposing any type of continuity on the map $T$. It is also good to mention at this point that complex valued $G$-metric space have close similarities with cone $G$-cone metric spaces( see\cite{me}) even though both spaces are very different. Moreover the rational contraction we considered is better applicable and understood when studied in the complex valued case.
\end{remark}

\bibliographystyle{amsplain}

\begin{thebibliography}{99}

\bibitem{me} A. Azam and N. Mehmood; \textit{Fixed point theorems for multivalued
mappings in $G$-cone metric spaces}, Journal of Inequalities and Applications 2013, 2013:354.



\bibitem{Gaba1} Y. U. Gaba; \textit{$\lambda$-sequences and fixed point theorems in $G$-metric spaces}, Journal of Nigerian Mathematical Society, Vol. 35, pp. 303-311, 2016.



\bibitem{Gaba5} Y. U. Gaba; \textit{New Contractive Conditions for Maps in $G$-metric Type Spaces}, Advances in Analysis, Vol. 1, No. 2, October 2016. 

\bibitem{Gaba6} Y. U. Gaba; \textit{Fixed point on ordered complex valued $G$-metric spaces}, in preparation.



\bibitem{s} S. K. Mohanta; 
\textit{Some Fixed Point Theorems in $G$-metric Spaces},
Analele \c{S}t.  Univ. Ovidius Constan\c{t}a, Vol. 20(1), 2012, 285--306




\bibitem{Mustafa} Z. Mustafa and B. Sims; \textit{A new approach to generalized metric spaces}, Journal of Nonlinear Convex Analysis, 7 (2006), 289--297.


\bibitem{mustafa3} Z. Mustafa, H. Obiedat, and F. Awawdeh; \textit{Fixed Point Theorem for Expansive Mappings
in $G$-Metric Spaces}, Int. J. Contemp. Math. Sciences, Vol. 5, 2010, no. 50, 2463 - 2472.


\bibitem{mustafa4} Z. Mustafa, H. Obiedat, and F. Awawdeh; \textit{Some fixed point theorem for mappings on a complete $G$- metric space}, Fixed Point Theory and Applications Volume 2008, Article ID 189870, 12 pages.





\bibitem{p} S. R. Patil; \textit{Expansion Mapping Theorems in
$G$-cone Metric Spaces}, Int. Journal of Math. Analysis, Vol. 6, 2012, no. 44, 2147 - 2158.





\bibitem{v} V. Sihag, R. K. Vats and C. Vetro; \textit{A fixed point theorem in $G$-metric spaces via $\alpha$-series}, Quaestiones Mathematicae Vol. 37 , Iss. 3,Pages 429-434, 2014.







\end{thebibliography}

\end{document}